\documentclass{amsart}

\usepackage[utf8]{inputenc}
\usepackage{color}
\usepackage{xcolor}
\usepackage{hyperref}
\usepackage{amssymb}
\usepackage{float}
\usepackage{tikz,tikz-cd}
\usepackage{forest}
\usepackage{nicefrac}
\usepackage{enumerate}
\usepackage[nameinlink, capitalize, noabbrev]{cleveref}
\usetikzlibrary{shapes.geometric,decorations.markings}
\usetikzlibrary{decorations.pathreplacing}
\usetikzlibrary{fit}
\usetikzlibrary{automata}
\usetikzlibrary{positioning}
\usetikzlibrary{intersections}
\tikzset{mytext/.style={font=\small, text=black}}
\tikzset{main node/.style={circle,fill=lime!30,draw,minimum size=0.5cm,inner sep=0pt},
            }

\usepackage{pst-node}
\usepackage{pst-tree}

\def\BALL[#1](#2){\rput[t](#2){}%
        \pscircle[fillstyle=solid,fillcolor=#1!40](#2){5pt}}
        
\psset{treesep=1.3,levelsep=1.5}

\hypersetup{
    colorlinks=true,
    linkcolor=teal,
    citecolor=magenta,
    }

\newtheorem{Theorem}{Theorem}

\newtheorem{Conjecture}{Conjecture}
\newtheorem{Corollary}[Conjecture]{Corollary}
\newtheorem{proposition}{Proposition}[section]
\newtheorem{lemma}[proposition]{Lemma}
\newtheorem{corollary}[proposition]{Corollary}
\newtheorem{theorem}[proposition]{Theorem}

\theoremstyle{definition}
\newtheorem{remark}[proposition]{Remark}

\numberwithin{equation}{section}

\title{Arboreal representations of linear groups}
\author{Jorge Fariña-Asategui}
\address{Jorge Fariña-Asategui: Centre for Mathematical Sciences, Lund University, 223 62 Lund, Sweden -- Department of Mathematics, University of the Basque Country UPV/EHU, 48080 Bilbao, Spain}
\email{jorge.farina\_asategui@math.lu.se}
\thanks{The author is supported by the Spanish Government, grant PID2020-117281GB-I00, partly with FEDER funds and by the Walter Gyllenberg Foundation from the Royal Physiographic Society of Lund.}

\begin{document}

\begin{abstract}
Abért and Virág conjectured in 2005 that any embedding of a linear group over a pro-$p$ domain into the group of $p$-adic automorphisms $W_p$ should be zero-dimensional. We prove their conjecture in greater generality, namely for embeddings of linear groups over any integral domain into the automorphism group of a bounded rooted tree.
\end{abstract}

\maketitle

\section{introduction}
\label{section: introduction}

Let $G$ be a profinite group admitting a filtration of open normal subgroups
$$G=G_0\ge G_1\ge G_2\ge \dotsb \ge G_n\ge \dotsb \ge \bigcap_{n\ge 0} G_n=1.$$
This filtration endows $G$ with a natural metric $d:G\times G\to [0,\infty)$ given by
$$d(g,h):=\inf_{n\ge 0}\{|G:G_n|^{-1}\mid gh^{-1}\in G_n\}.$$
We may define a Hausdorff dimension $\mathrm{hdim}_G(\cdot)$ for the Borel subsets of the metric space $(G,d)$. If $H\le_c G$ is a closed subgroup, then its Hausdorff dimension coincides with its lower-box dimension by \cite[Theorem 2.4]{BarneaShalev}. In other words
\begin{align}
    \label{align: lower box dimension}
    \mathrm{hdim}_G(H)=\liminf_{n\to\infty}\frac{\log|HG_n:G_n|}{\log|G:G_n|}=\liminf_{n\to\infty}\frac{\log|H:H\cap G_n|}{\log|G:G_n|}.
\end{align}
If the limit in \cref{align: lower box dimension} exists, we say that $H$ has \textit{strong Hausdorff dimension} in $G$.

A classical example fitting into the above description is the automorphism group $\mathrm{Aut}~T$ of a spherically homogeneous rooted tree $T$. In this case, a filtration of open normal subgroups is given by the \textit{level-stabilizers}, where the $n$th level stabilizer~$\mathrm{St}(n)$ is simply the pointwise stabilizer of the vertices at the $n$th level $\mathcal{L}_n$ of~$T$.

If $T$ is the $p$-adic tree, we write $W_p$ for a Sylow pro-$p$ subgroup of $\mathrm{Aut}~T$, and call it the \textit{group of $p$-adic automorphisms}. Note that $W_p$ is simply the iterated wreath product corresponding to a $p$-cycle in $\mathrm{Sym}(p)$. In their seminal paper \cite{AbertVirag}, Abért and Virág made the following remarkable conjecture \cite[Conjecture 8.4]{AbertVirag}:

\begin{Conjecture}[Abért and Virág 2005]
\label{conjecture: linear groups}
    Let $R$ be a commutative pro-$p$ ring and let $G\le \mathrm{GL}_n(R)$ be a linear group over $R$. Then, for any embedding of $G$ into $W_p$ the image of $G$ has zero Hausdorff dimension.
\end{Conjecture}

\cref{conjecture: linear groups} has been open for two decades, even for local fields. The main goal of this paper is to establish its veracity. In fact, we prove a stronger result.

A rooted tree $T$ is said to be \textit{bounded} if the vertices at the same level of $T$, all have the same number of immediate descendants and this number is at most $M$ for every level, for some $M\ge 2$. For us, a ring $R$ is an \textit{integral domain} if it is commutative, it has a multiplicative identity and it has no non-trivial zero-divisors (i.e. the zero ideal is prime). The main result of this paper reads as follows:

\begin{Theorem}
\label{Theorem: linear groups are zero-dimensional}
    Let $R$ be an integral domain  and let $G\le \mathrm{GL}_n(R)$ be a linear group over $R$. Let $T$ be a bounded tree. Then, for any embedding of $G$ into $\mathrm{Aut}~T$ the image of $G$ has strong zero Hausdorff dimension.
\end{Theorem}

In particular, when $R$ is taken to be a pro-$p$ domain (i.e. a complete, Noetherian, local domain $R$ with finite residue field of characteristic~$p$) and $T$ the $p$-adic tree in \cref{Theorem: linear groups are zero-dimensional}, one obtains \cref{conjecture: linear groups}.

The proof of \cref{Theorem: linear groups are zero-dimensional} consists of three steps. First, we show that almost all the projections on the tree of a positive-dimensional subgroup of $\mathrm{Aut}~T$ are weakly branch; see \cref{proposition: rists of positive dimensional subgroups}. Secondly, we follow the approach of Abért in \cite{AbertGraph} and study the non-commuting representations of graphs in linear groups over a general integral domain. Lastly, we prove a lifting property, which allows us to lift a non-commuting representation of a graph from a projection of a group on the tree to the original group.

We believe that the individual tools developed in some of these steps are of independent interest. Let $G$ be a topological group. By an \textit{arboreal representation} of $G$ on a bounded rooted tree~$T$ we shall mean a continuous group homomorphism $\rho:G\to \mathrm{Aut}~T$. In \cref{section: structural properties}, we use \cref{proposition: rists of positive dimensional subgroups} to prove that groups satisfying a law do not admit positive-dimensional arboreal representations:

\begin{Corollary}
    \label{corollary: group laws}
    Let $G$ be a topological group satisfying a law and $T$ a bounded tree. Then, for any arboreal representation of $G$ on $T$, the image of $G$ has strong zero Hausdorff dimension.
\end{Corollary}

A version of \cref{corollary: group laws} for self-similar groups was obtained by the author in~\cite{QuestionAbertVirag}. \cref{corollary: group laws} was proposed in \cite[page 186]{AbertVirag} as a possible generalization to the zero-dimensionality of solvable subgroups of $W_p$ proved by Abért and Virág in \cite[Theorem 5]{AbertVirag}. \cref{corollary: group laws} completely settles this problem. Further applications of \cref{proposition: rists of positive dimensional subgroups} to the subgroup structure of positive-dimensional subgroups are discussed at the end of \cref{section: structural properties}.

\cref{conjecture: linear groups} is not only relevant to group theory, it also has a strong connection with arithmetic geometry. One of the central problems in modern arithmetic geometry is a notable conjecture of Fontaine and Mazur in relation to the $p$-adic Galois representations of Galois groups of number fields. If $G_{K,S}$ denotes the maximal pro-$p$ extension of a number field $K$ unramified outside the primes above a finite set of primes $S$ (with $S$ not containing any prime above $p$) the Fontaine-Mazur conjecture \cite{FontaineMazur} says that any $p$-adic Galois representation 
$$\rho:G_{K,S}\to \mathrm{GL}_n(\mathbb{Z}_p)$$
has finite image; see \cite{Calegari,Kisin,Pan,SkinnerWiles} for some progress on this well-known conjecture. Boston proposed in \cite{Boston} a further generalization of the Fontaine-Mazur conjecture for linear representations over a pro-$p$ domain.

Recall that a \textit{just-infinite} profinite group is a profinite group whose non-trivial normal closed subgroups are all open. As discussed by Boston in \cite[Section 3]{BostonNH}, the critical cases for the Fontaine-Mazur conjecture concern precisely unramified pro-$p$ extensions of number fields with just-infinite pro-$p$ Galois groups. We have the following immediate corollary to \cref{Theorem: linear groups are zero-dimensional}, which suggests a strong relation between \cref{Theorem: linear groups are zero-dimensional} and Boston's generalization of the Fontaine-Mazur conjecture:

\begin{Corollary}
\label{corollary: just infinite}
    Let $T$ be a bounded tree and $G\le_c \mathrm{Aut}~T$ a closed just-infinite subgroup with positive Hausdorff dimension in $\mathrm{Aut}~T$. Then, for any integral domain $R$, any linear representation
    $$\rho:G\to \mathrm{GL}_n(R)$$
    has finite image.
\end{Corollary}

\cref{corollary: just infinite} suggests that arboreal representations may represent certain classes of groups better than classical $p$-adic Galois representations. In particular, \cref{corollary: just infinite} holds for any pro-$p$ domain; thus, Boston's generalization of the Fontaine-Mazur conjecture holds for unramified pro-$p$ extensions of number fields with just-infinite pro-$p$ Galois groups admitting a positive-dimensional arboreal representation.

Arboreal representations have been studied successfully, since the seminal work of Odoni in \cite{Odoni1, Odoni2}, for Galois groups of iterates of polynomials and rational functions; cf. \cite{Bridy, JorgeSantiFPP, JonesAMS, JonesComp, JonesLMS, Juul}. Arboreal representations may be seen as a far-reaching generalization of Galois representations on Tate modules of abelian varieties. Motivated by Serre's celebrated open image theorem for elliptic curves without complex multiplication, Jones conjectured in \cite{JonesArboreal} that the image of arboreal representations of rational functions of degree two are of finite index in $\mathrm{Aut}~T$, if none of the obvious obstructions occurs (such as the rational function being post-critically finite). Even when these obstructions occur, the image of an arboreal representation is expected to be large (in the sense that it has positive Hausdorff dimension in $\mathrm{Aut}~T$) except for powering maps and Chebyshev polynomials. In fact, the recent results of Ferraguti, Ostafe and Zannier in \cite{Ferraguti} suggest that these are indeed the only cases when the image of an arboreal representation is small.

The evidence on positive-dimensionality of arboreal representations of certain Galois groups together with \cref{Theorem: linear groups are zero-dimensional} and \cref{corollary: just infinite}, suggest that a better understanding of the arboreal representations of the Galois groups $G_{K,S}$ might provide us with new tools to tackle both the Fontaine-Mazur conjecture and Boston's generalization.

\subsection*{\textit{\textmd{Organization}}}  In \cref{section: structural properties} we prove the first step towards the proof of \cref{Theorem: linear groups are zero-dimensional} and discuss some further applications. In \cref{section: representations}, we discuss non-commuting representations of graphs into linear groups. Then, we prove a lifting result allowing us to lift non-commuting representations from the projection of a group on a subtree to the original group. Finally, we put all the previous results together and prove \cref{Theorem: linear groups are zero-dimensional}.

\subsection*{\textit{\textmd{Notation}}} We write $H\le_c G$ to denote a closed subgroup of $G$.  We write $T$ for a bounded rooted tree, $m_n$ for the number of descendants of each vertex at $\mathcal{L}_n$ and~$N_n$ for the total number of vertices at $\mathcal{L}_n$. We also write $|v|$ for the level at which a vertex $v\in T$ lies. Exponential notation will be used for group actions on $T$ and on its boundary $\partial T$. We also write $\underline{\mathrm{dim}}_T(G),\overline{\mathrm{dim}}_T(G)$ and $\mathrm{hdim}_T(G)$ for the lower-box, upper-box and Hausdorff dimension of $G\le_c \mathrm{Aut}~T$ respectively, where the subscript $T$ indicates that the corresponding dimension of $G$ is considered as a subgroup of $\mathrm{Aut}~T$. Lastly, the expression \textit{for almost all} will be employed throughout the paper to mean \textit{for all but finitely many}.

\section{Structural properties of positive-dimensional subgroups}
\label{section: structural properties}

In this section, we first fix some notation regarding group actions on rooted trees that we shall use in the rest of the paper. We then prove the first step towards establishing \cref{Theorem: linear groups are zero-dimensional}, namely that almost all the projections on the tree of a positive-dimensional subgroup are weakly branch. We conclude the section with some applications of this first result.

\subsection{Groups acting on rooted trees}

Let $T$ be a spherically homogeneous rooted tree. We shall assume that $T$ is bounded, i.e. there exists some $M\ge 2$ such that $m_k\le M$ for all $k\ge 0$. For any $v\in T$, we denote by $T_v$ the subtree rooted at the vertex $v$, which is again a spherically homogeneous bounded rooted tree. Note that~$T_v$ and $T_w$ are isomorphic as graphs when $v$ and $w$ are both at the same level, as~$T$ is spherically homogeneous. Hence, we shall write $u\in T_v=T_w$ when the same vertex $u$ (up to this isomorphism) needs to be considered in both subtrees~$T_v$ and~$T_w$. We may identify vertices in $T$ with finite words $v_1v_2\dotsb v_n$ such that $v_1\in T$ and $v_{i+1}\in T_{v_i}$ for $1\le i \le n-1$. The boundary $\partial T$ of $T$ is the set of infinite rooted paths in $T$. We shall write $v\in \gamma$ for some $\gamma\in \partial T$ if the path $\gamma$ passes though $v$.

We write $\mathrm{Aut}~T$ for the group of automorphisms of $T$ fixing the root. For a subgroup $G\le \mathrm{Aut}~T$, we write $\mathrm{St}_G(n)$ and $\mathrm{st}_G(v)$ for the pointwise stabilizer in $G$ of the $n$th level of $T$ and for the stabilizer in~$G$ of the vertex $v\in T$ respectively. The reader should note that when the subscript in the stabilizers is omitted, they refer to the stabilizers in the full automorphism group of the appropriate (sub)tree.

For each $g\in \mathrm{Aut}~T$ and $v\in T$, there exists a unique $g|_v\in \mathrm{Aut}~T_v$ such that 
$$(vw)^g=v^gw^{g|_v}$$
for any $w\in T_v$. We define the homomorphism $\varphi_v:\mathrm{st}(v)\to \mathrm{Aut}~T_v$ via $g\mapsto g|_v$. We shall write 
$$G_v:=\varphi_{v}(\mathrm{st}_G(v))$$
in the remainder of the paper to simplify notation. For each $k\ge 1$, we further define the isomorphism $\psi_k:=\prod_{v\in\mathcal{L}_k}\varphi_v:\mathrm{St}(k)\to \prod_{v\in \mathcal{L}_k} \mathrm{Aut}~T_{v}$, i.e. the isomorphism given by
$$g\mapsto (g|_{v_1},\dotsc, g|_{v_{N_k}}).$$
The isomorphism $\psi_k$ yields the following equality on the logarithmic indices
\begin{align}
\label{align: equality of log ind}
\log|\mathrm{St}(k):\mathrm{St}(k+n)|&=\sum_{v\in \mathcal{L}_k}\log|\mathrm{Aut}~T_v:\mathrm{St}(n)|=N_k\cdot\log|\mathrm{Aut}~T_v:\mathrm{St}(n)|
\end{align}
for every $n\ge 1$. Restricting the isomorphism $\psi_k$ to a subgroup $G\le \mathrm{Aut}~T$ yields an embedding
$$\psi_k:\mathrm{St}_G(k)\hookrightarrow \prod_{v\in \mathcal{L}_k}G_{v}.$$
This last embedding yields the following upper bound on the logarithmic indices
\begin{align}
\label{align: upper bound of log ind}
\log|\mathrm{St}_G(k):\mathrm{St}_G(k+n)|&\le \sum_{v\in \mathcal{L}_k}\log|G_v:\mathrm{St}_{G_v}(n)| \end{align}
for every $n\ge 1$.

The isomorphism $\psi_k$ can be lifted to an isomorphism (which we shall also denote $\psi_k$ by a slight abuse of notation) $\psi_k:\mathrm{Aut}~T\to\big(\prod_{v\in \mathcal{L}_k} \mathrm{Aut}~T_{v}\big) \rtimes (\mathrm{Aut}~T/\mathrm{St}(k))$, given by
$$g\mapsto (g|_{v_1},\dotsc, g|_{v_{N_k}})g|^k,$$
where $g|^k\in \mathrm{Aut}~T/\mathrm{St}(k)$ describes the action of $g$ on the $k$th level of $T$ by identifying the quotient $\mathrm{Aut}~T/\mathrm{St}(k)$ with a subgroup of $\mathrm{Sym}(\mathcal{L}_k)$.

For a closed subgroup $G\le_c \mathrm{Aut}~T$, its lower-box dimension $\underline{\mathrm{dim}}_T(G)$ in $\mathrm{Aut}~T$ and its upper-box dimension $\overline{\mathrm{dim}}_T(G)$ in $\mathrm{Aut}~T$ are given by
$$\underline{\mathrm{dim}}_T(G)=\liminf_{n\to\infty}\frac{\log|G:\mathrm{St}_G(n)|}{\log|\mathrm{Aut}~T:\mathrm{St}(n)|}$$
and 
$$\overline{\mathrm{dim}}_T(G)=\limsup_{n\to\infty}\frac{\log|G:\mathrm{St}_G(n)|}{\log|\mathrm{Aut}~T:\mathrm{St}(n)|}$$
respectively. Recall the standard inequalities of fractal dimensions:
$$0\le \underline{\mathrm{dim}}_T(G)\le \mathrm{hdim}_T(G)\le \overline{\mathrm{dim}}_T(G)\le 1.$$
We shall say that $G$ has \textit{strong Hausdorff dimension} in $\mathrm{Aut}~T$ if the above three fractal dimensions coincide (the reader should note that this definition coincides with the one given in the introduction). By \cite[Theorem 2.4]{BarneaShalev}, we always have the equality $\underline{\mathrm{dim}}_T(G)= \mathrm{hdim}_T(G)$.

For a subgroup $G\le \mathrm{Aut}~T$, we define for each $v\in T$ the corresponding \textit{rigid vertex stabilizer} $\mathrm{rist}_G(v)$ as the subgroup consisting of the elements in $G$ which fix every vertex not in $T_v$. Rigid vertex stabilizers of distinct vertices at the same level commute and have trivial intersection; hence, we shall define for each level $n$ the \textit{rigid level stabilizer} $\mathrm{Rist}_G(n)$ as the direct product $\mathrm{Rist}_G(n):=\prod_{v\in \mathcal{L}_n}\mathrm{rist}_G(v)$. Note that $\mathrm{Rist}_G(n)$ is a normal subgroup of $G$. A subgroup $G\le \mathrm{Aut}~T$ is said to be \textit{level-transitive} if it acts transitively on every level of $T$. A level-transitive subgroup whose rigid level stabilizers are non-trivial (of finite index in $G$) is called \textit{weakly branch} (respectively \textit{branch}). It is easy to see that the rigid stabilizers of weakly branch groups are in fact infinite.

\subsection{Projections of positive-dimensional groups}

It was shown by the author in \cite{QuestionAbertVirag} that, unless one assumes self-similarity, there are examples of level-transitive positive-dimensional subgroups of $\mathrm{Aut}~T$ which are not weakly branch. However, in the examples constructed in \cite{QuestionAbertVirag}, the projections on the first level of the tree are all weakly branch. We show that under the assumption that $T$ is bounded, all level-transitive positive-dimensional subgroups of $\mathrm{Aut}~T$ satisfy a similar property:

\begin{theorem}
    \label{proposition: rists of positive dimensional subgroups}
    Let $T$ be a bounded rooted tree and $G\le_c \mathrm{Aut}~T$ a level-transitive closed subgroup with positive upper-box dimension in $\mathrm{Aut}~T$. Then, there exists some $N\ge 1$ such that for every $n\ge N$ and every vertex $v\in \mathcal{L}_n$ the projections
    $$G_v\le \mathrm{Aut}~T_v$$
    are all weakly branch.
\end{theorem}
\begin{proof}
    First, note that if $G_v$ is weakly branch then, for any vertex $w\in T_v$, the projection $G_{vw}$ is also weakly branch. Indeed, all the projections are level-transitive by \cite[Proposition 4.2]{RestrictedSpectra} and we further have
    $$G_{vw}=\varphi_{vw}(\mathrm{st}_G(vw))=\varphi_w(\varphi_v(\mathrm{st}_G(vw)))=\varphi_w(\mathrm{st}_{G_v}(w)).$$
    Hence, for any $u\in T_{vw}$ we get
    $$\mathrm{rist}_{G_{vw}}(u)=\varphi_w(\mathrm{rist}_{G_v}(wu))\ne 1$$
    as $\varphi_w|_{\mathrm{rist}_{G_v}(w)}$ is injective.

    Let us assume by contradiction that there does not exist such a level $N$. Then, by the above observation, there must exist an infinite path $\gamma\in \partial T$ such that for every $v\in \gamma$, the projection $G_v$ is not weakly branch. In fact, by the level-transitivity of~$G$, this holds for every infinite path in $\partial T$, as for any $v\in T$ and $g\in G$ we have~$G_v$ is weakly branch if and only if $G_{v^g}$ is weakly branch.

    As the upper-box dimension does not change when passing to finite-index subgroups, for every $v\in T$ and every $k\ge 1$, we obtain the following inequality for the upper-box dimensions as a direct application of \textcolor{teal}{(}\ref{align: equality of log ind}\textcolor{teal}{)} and \textcolor{teal}{(}\ref{align: upper bound of log ind}\textcolor{teal}{)}:
    \begin{align}
        \label{align: first bound on dim}
        \overline{\mathrm{dim}}_{T_v}(G_v)&=\limsup_{n\to\infty}\frac{\log|G_v:\mathrm{St}_{G_v}(n)|}{\log|\mathrm{Aut}~T_v:\mathrm{St}(n)|}\\
        &=\limsup_{n\to\infty}\frac{\log|\mathrm{St}_{G_v}(k):\mathrm{St}_{G_v}(k+n)|}{\log|\mathrm{St}(k):\mathrm{St}(k+n)|}\nonumber\\
        &\le \limsup_{n\to\infty}\frac{\sum_{w\in \mathcal{L}_k}\log|G_{vw}:\mathrm{St}_{G_{vw}}(n)|}{N_k\cdot\log|\mathrm{Aut}~T_{vw}:\mathrm{St}(n)|}\nonumber\\
        &\le \frac{1}{N_k}\sum_{w\in \mathcal{L}_k}\limsup_{n\to\infty}\frac{\log|G_{vw}:\mathrm{St}_{G_{vw}}(n)|}{\log|\mathrm{Aut}~T_{vw}:\mathrm{St}(n)|}\nonumber\\
        &\le \frac{1}{N_k} \sum_{w\in \mathcal{L}_k}\overline{\mathrm{dim}}_{T_{vw}}(G_{vw})\nonumber\\
        &\le \max_{w\in \mathcal{L}_k}\{\overline{\mathrm{dim}}_{T_{vw}}(G_{vw})\} \nonumber\\
        &=\overline{\mathrm{dim}}_{T_{vw}}(G_{vw})\nonumber
    \end{align}
    for any $w\in\mathcal{L}_k\subseteq T_v$, as $G_v$ is level-transitive.
    
    We may assume without loss of generality that $\mathrm{Rist}_{G_v}(1)=1$ for infinitely many $v\in \gamma$, as if $\mathrm{Rist}_{G_v}(k)=1$ for some $k\ge 2$, by projecting $G_v$ further in $\gamma$ we always find a vertex $vu\in \gamma$ such that $\mathrm{Rist}_{G_{vu}}(1)=1$. 
    
    Since $\mathrm{Rist}_{G_v}(1)=1$, the embedding 
    $$\mathrm{St}_{G_v}(1)\hookrightarrow \prod_{x=1}^{m_{|v|}}\mathrm{Aut}~T_{vx},$$ 
    has a restricted image. Thus, projecting this image into the first $m_{|v|}-1$ components yields a monomorphism, i.e.
    $$\mathrm{St}_{G_v}(1)\hookrightarrow \prod_{x=1}^{m_{|v|}-1}\mathrm{Aut}~T_{vx}.$$
    Then, arguing as in \textcolor{teal}{(}\ref{align: first bound on dim}\textcolor{teal}{)}, we obtain the inequality on the upper-box dimensions:
    \begin{align}
        \label{align: second inequality dim}
        \overline{\mathrm{dim}}_{T_v}(G_v)&\le \overline{\mathrm{dim}}_{T_{vx}}(G_{vx})\cdot \left(\frac{m_{|v|}-1}{m_{|v|}}\right)
    \end{align}
    for any $x\in \mathcal{L}_1\subset T_v$. In particular, we may choose $x$ such that $vx\in \gamma$.
    
    If, for $n\ge 1$, we write $V_n\subseteq \gamma$ for those vertices in $\gamma$ up to level $n$ for which the condition $\mathrm{Rist}_{G_v}(1)=1$ is satisfied, we further obtain the inequality
    \begin{align}
        \label{align: third inequality dim}
        \overline{\mathrm{dim}}_{T}(G)\le \overline{\mathrm{dim}}_{T_{\gamma\cap \mathcal{L}_n}}(G_{\gamma\cap \mathcal{L}_n})\cdot \prod_{v\in V_n}\left(\frac{m_{|v|}-1}{m_{|v|}}\right)\le \prod_{v\in V_n}\left(\frac{m_{|v|}-1}{m_{|v|}}\right)
    \end{align}
    for every $n\ge 1$, as an application of the inequalities in \textcolor{teal}{(}\ref{align: first bound on dim}\textcolor{teal}{)} and \textcolor{teal}{(}\ref{align: second inequality dim}\textcolor{teal}{)} along the path~$\gamma$ up to level $n$. Since the condition $\mathrm{Rist}_{G_v}(1)=1$ is fulfilled for infinitely many $v\in \gamma$, by taking upper limits in the inequality in \textcolor{teal}{(}\ref{align: third inequality dim}\textcolor{teal}{)} and using that $m_{k}\le M$ for all $k\ge 1$ (as $T$ is bounded) we get
    $$\overline{\mathrm{dim}}_{T}(G)\le \limsup_{n\to\infty}\prod_{v\in V_n}\left(\frac{m_{|v|}-1}{m_{|v|}}\right)\le \lim_{n\to\infty} \left(\frac{M-1}{M}\right)^n=0.$$
    This contradicts the assumption that $G$ had positive upper-box dimension in $\mathrm{Aut}~T$ and concludes the proof.  
\end{proof}

\begin{remark}
    Level-transitivity of $G$ is only used mildly in the proof of \cref{proposition: rists of positive dimensional subgroups}. Indeed, we only use it to ensure that the projections $G_v$ are level-transitive (as level-transitive is part of the definition of weakly branch), and to obtain the result for almost all projections. If one drops level-transitivity, then the proof of \cref{proposition: rists of positive dimensional subgroups} simply yields an infinite path $\gamma\in \partial T$ such that for almost all vertices in $\gamma$, the projections $G_v$ have non-trivial rigid level stabilizers at every level. This essentially means that the induced action of $G_v$ on the boundary $\partial T_v$ is a micro-supported action; see \cite{MicroSupported1, MicroSupported2}.
\end{remark}

\subsection{Some applications}

An immediate consequence of \cref{proposition: rists of positive dimensional subgroups} is that a level-transitive closed subgroup $G\le_c \mathrm{Aut}~T$, satisfying that for infinitely many vertices $v\in T$ the projections $G_v$ are not weakly branch, must have zero upper-box dimension in $\mathrm{Aut}~T$. Equivalently, $G$ must have strong zero Hausdorff dimension in $\mathrm{Aut}~T$.

\cref{proposition: rists of positive dimensional subgroups} is enough to show that groups satisfying a law do not admit positive-dimensional arboreal representations on bounded rooted trees:

\begin{proof}[Proof of \cref{corollary: group laws}]
    If $G$ satisfies a law, any arboreal representation of $G$ satisfies a law. Let us write $\rho(G)$ for the image of an arboreal representation on a bounded rooted tree $T$. Now, any subgroup of $\rho(G)$ satisfies the same law as $G$. In particular, for every $v\in T$, the subgroup $\mathrm{St}_{\rho(G)}(v)$ satisfies a law. Furthermore, the projection $\rho(G)_v$ satisfies a law, so it cannot be weakly branch by \cite[Theorem 1]{Abert}. Therefore, by \cref{proposition: rists of positive dimensional subgroups}, the group $\rho(G)$ must have strong zero Hausdorff dimension in $\mathrm{Aut}~T$.
\end{proof}

A second consequence of \cref{proposition: rists of positive dimensional subgroups} is that a group  admitting a positive-dimensional arboreal representation on a bounded rooted tree must have finite center:

\begin{corollary}
\label{corollary: center}
    Let $G\le \mathrm{Aut}~T$ be a level-transitive subgroup whose closure in $\mathrm{Aut}~T$ has positive upper-box dimension in $\mathrm{Aut}~T$. Then the center of $G$ is finite.
\end{corollary}
\begin{proof}
    Since $G$ is residually finite $G$ embeds into its closure $\overline{G}$, and the center of~$G$, which we denote $Z(G)$, is contained in the center of $\overline{G}$, as $Z(G)N/N\le Z(G/N)$ for any normal subgroup $N\trianglelefteq G$. Thus, we may assume without loss of generality that~$G$ is closed in $\mathrm{Aut}~T$.
    
    Let us assume that $n\ge N$, where $N$ is given by \cref{proposition: rists of positive dimensional subgroups}. Let $g\in \mathrm{St}_G(n)$ be non-trivial. Then, there is a vertex $v$ at the $n$th level of $T$ such that $\varphi_v(g)\ne 1$. Since the projection $G_v$ is weakly branch by \cref{proposition: rists of positive dimensional subgroups}, the proof of \cite[Proposition~2.6]{BartholdiHausdorff} shows that the center of any subgroup of finite index in $G_v$ is trivial. Thus, one can find an element $h\in \mathrm{St}_G(n)$ such that $[\varphi_v(g),\varphi_v(h)]\ne 1$. Therefore $[g,h]\ne 1$, as $\psi_n=\prod_{v\in\mathcal{L}_n}\varphi_v$ and $\psi_n$ is injective, so $g\notin Z(G)$. Hence
    $$Z(G)\cap \mathrm{St}_G(n)=1,$$
    and since $Z(G)$ is a normal subgroup of $G$ and $\mathrm{St}_G(n)$ is normal too,  we get
    $$\langle \mathrm{St}_G(n),Z(G)\rangle=\mathrm{St}_G(n)\times Z(G).$$
    In particular $Z(G)$ is finite as it embeds in the finite group $G/\mathrm{St}_G(n)$.
\end{proof}

Note that \cref{corollary: center} is sharp in the sense that there exist level-transitive closed subgroups of $\mathrm{Aut}~T$ with non-trivial finite centers and positive Hausdorff dimension; see \cite[Proposition 3.1]{QuestionAbertVirag} for such examples.

\textcolor{teal}{Corollaries} \ref{corollary: group laws} and \ref{corollary: center} are enough to show that many groups do not admit positive-dimensional arboreal representations on bounded rooted trees. If $F$ is a local field, by the work of Breuillard and Gelander in \cite{TitsAlternative}, a linear group $G\le \mathrm{GL}_n(F)$ over~$F$ satisfies a topological analogue of the Tits alternative: $G$ either contains an open solvable subgroup or a dense free subgroup. In the former, by \cref{corollary: group laws}, the group $G$ does not admit a positive-dimensional arboreal representation on a bounded rooted tree. In the latter, we shall need a different approach based on non-commuting representations of graphs, as done in the next section. However, \cref{corollary: center} already shows, for instance, that the full group $\mathrm{GL}_n(F)$ does not admit  positive-dimensional arboreal representations on a bounded rooted tree, as its center is infinite.

\section{Non-commuting representations of graphs}
\label{section: representations}

Let $(V,E)$ be an undirected graph with no loops or multiple edges. A map $f:V\to G$ from the set of vertices $V$ to a group $G$ is said to be a \textit{non-commuting representation} of $(V,E)$ in $G$, if for all $v,w\in V$ we have 
$$(v,w)\in E\quad\text{if and only if}\quad f(v)\cdot f(w)\ne f(w)\cdot f(v).$$

For $n\ge 1$, let $V_n$ be the graph consisting of $2n$ vertices where each vertex has degree exactly 1, i.e. $V_n$ is the disjoint union of $n$ copies of $V_1$, where $V_1$ consists of two vertices and a unique edge joining both vertices; see \cref{figure: graph Vn} below.

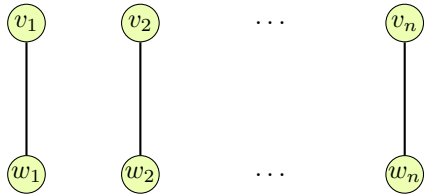
\begin{figure}[H]
        \centering
        \begin{tikzpicture}
    \node[main node] (1) {$v_1$};
    \node[main node] (2) [below = 1.5cm of 1]  {$w_1$};
    \node[main node] (3) [right = 1cm  of 1] {$v_2$};
    \node[main node] (4) [right = 1cm  of 2] {$w_2$};
    \node[main node] (5) [right = 3cm  of 3] {$v_n$};
    \node[main node] (6) [below = 1.5cm  of 5] {$w_n$};

    \path[draw,thick]
    (1) edge node {} (2)
    (3) edge node {} (4)
    (5) edge node {} (6);
\node at ($(3)!.5!(5)$) {\ldots};
\node at ($(4)!.5!(6)$) {\ldots};
\end{tikzpicture}
        \caption{The graph $V_n$.}
        \label{figure: graph Vn}
\end{figure}

\begin{remark}
    A non-commuting representation of $V_n$ in $G$ corresponds to no more than a direct product of $n$ non-abelian 2-generated subgroups in $G$. Weakly branch groups are precisely groups admitting such direct products as subgroups for every $n\ge 1$, while linear groups over a field do not admit such direct products for arbitrarily big $n\ge 1$. This was used by Abért in \cite[Corollary 7]{AbertGraph} to prove that weakly branch groups are not linear over a field. We shall extend the approach of Abért from fields to general integral domains.
\end{remark}

\subsection{Linear groups over an integral domain}

Let us fix an integral domain $R$. We shall assume in the following that $\mathrm{GL}_k(R)$ is infinite, as otherwise \cref{Theorem: linear groups are zero-dimensional} holds trivially.

Let us denote by $\mathrm{mat}_R(n)$ the minimal $k$ such that $\mathrm{GL}_k(R)$ admits a non-commuting representation of $V_n$. By \cite[Proposition 5]{AbertGraph}, this is equivalent to $\mathrm{GL}_k(R)$ admitting a non-commuting representation of any graph of size $n$, which was the original definition of $\mathrm{mat}_R(n)$ given by Abért in \cite{AbertGraph}. Abért obtained in \cite[Theorem~3]{AbertGraph} a lower bound for $\mathrm{mat}_R(n)$ when $R$ is further assumed to be a field. Here, we obtain a lower bound for $\mathrm{mat}_R(n)$ when $R$ is an integral domain, arguing as Abért in \cite[Theorem 3]{AbertGraph} (the upper bound in \cite[Theorem 3]{AbertGraph} also holds for integral domains, but since we do not need it here we omit it). Note that the lower bound we obtain is slightly better than the one stated in \cite[Theorem 3]{AbertGraph}, as instead of treating $V_n$ as a graph of size $2n$, one can first represent $V_n$ in $\mathrm{GL}_k(R)$ and later apply the equivalence in \cite[Proposition 5]{AbertGraph}. A sharp linear lower bound (for fields) has been recently obtained by Kionke and Schesler in \cite[Theorem 2]{KionkeSchesler}. For our purposes, any lower bound diverging to infinity on $n$ suffices.

\begin{lemma}
\label{lemma: minimal number of representations}
   For every integral domain $R$  we have
   $$\mathrm{mat}_R(n)\ge \sqrt{n}.$$
\end{lemma}
\begin{proof}
    Note that since $R$ is commutative and unital $\mathrm{M}_k(R)$ is a free $R$-module of rank $k^2$ for every $k\ge 1$.
    
    Now, let us consider $k= \mathrm{mat}_R(n)$. Then $\mathrm{GL}_k(R)$ admits a non-commuting representation of $V_n$, i.e. we can find $a_1,b_1,\dotsc,a_n,b_n\in \mathrm{GL}_k(R)$ such that for every $1\le i,j\le n$ we have the commuting relations
    $$a_ia_j=a_ja_i\quad \text{and}\quad b_ib_j=b_jb_i,$$
    and both commuting and non-commuting relations
    $$a_ib_j= b_ja_i\quad \text{if and only if}\quad i\ne j.$$
    This choice of the $a_i$'s and the $b_i$'s guarantees that the set $\{a_1,\dotsc,a_n\}\subseteq \mathrm{M}_k(R)$ is linearly independent over $R$. In fact, assume by contradiction that the set $\{a_1,\dotsc,a_n\}$ is not linearly independent over $R$. Then, there exists $\lambda_1,\dotsc, \lambda_n\in R$ such that
    $$\lambda_j a_j=\sum_{i\ne j}\lambda_ia_i.$$
    However, then we obtain
    \begin{align*}
        \lambda_ja_jb_j&=\left(\sum_{i\ne j}\lambda_ia_i\right)b_j=\sum_{i\ne j}\lambda_ia_ib_j=\sum_{i\ne j}b_j\lambda_ia_i=b_j\left(\sum_{i\ne j}\lambda_ia_i\right)=b_j\lambda_ja_j\\
        &=\lambda_jb_ja_j,
    \end{align*}
    which yields a contradiction as
    $$\lambda_ja_jb_j=\lambda_jb_ja_j\quad\text{if and only if}\quad a_jb_j=b_ja_j$$
    since $R$ is an integral domain. Hence $\{a_1,\dotsc,a_n\}$ is linearly independent over $R$ and we get
    \begin{align*}
    \mathrm{rank}_R(\mathrm{M}_k(R))&=k^2\ge n\implies \mathrm{mat}_R(n)\ge \sqrt{n}.\qedhere
    \end{align*}
\end{proof}

\subsection{Lifing non-commuting representations}

    We would like to construct non-commuting representations of $V_n$ in level-transitive positive-dimensional groups for arbitrarily large $n\ge 1$. This is not easy to do a priori. However, by \cref{proposition: rists of positive dimensional subgroups}, we know that almost all of the projections of such a group on the tree are weakly branch, so they admit non-commuting representations of $V_n$ for arbitrarily large $n\ge 1$ by \cite[Corollary~7]{AbertGraph}. Therefore, it will be enough to find a way of lifting these non-commuting representations of $V_n$ from the projections of a group on the tree to the original group. The lifting lemma below will be the key to do so. Before stating and proving the lifting lemma, we fix some notation and terminology.
    
    For any $K\le H\le G$ chain of subgroups, we denote by $K^H\le H$ the normal closure of $K$ in $H$. Given a subgroup
    $$H\le G_1\times\dotsb \times G_k$$
    of a direct product, we say that $H$ is a \textit{subdirect product} (of $G_1\times\dotsb \times G_k$) if the projection of $H$ to each component $G_i$ is surjective.

\begin{lemma}[Lifting lemma]
    \label{lemma: lifting lemma}
    Let $H,L\le U$ be two non-trivial subgroups of some group $U$ such that for some $k\ge 0$ there exist  $h_0,h_1,\dotsc,h_k\in H$ satisfying that
    $$H\le L^{h_0}\times L^{h_1}\times\dotsb\times  L^{h_{k}},$$
    is a subdirect product. Then, there exists $i\in \{0,1,\dotsc,k\}$ maximal, for which there is a non-trivial normal subgroup $1\ne N\trianglelefteq L^{h_0}$ such that for each $a\in N$, there exist unique $b_1,\dotsc,b_{k-i}$ satisfying that
    $$(a,1,\overset{i}{\ldots},1,b_1,\dotsc,b_{k-i})\in H\le  L^{h_0}\times L^{h_1}\times\dotsb\times  L^{h_{k}}.$$
\end{lemma}
\begin{proof}
    By conjugating with $h_0^{-1}$, we may assume without loss of generality that $h_0=1$. For $k=0$, the statement is trivially satisfied with $i=0$ and $N=L$. Let us assume by induction that the statement holds for $k\ge 0$ and let us prove it for $k+1$. Let us consider the projection 
    $$\pi:(L\times L^{h_1}\times\dotsb \times L^{h_{k}})\times L^{h_{k+1}}\to (L\times L^{h_1}\times\dotsb\times  L^{h_{k}})$$
    restricted to $H$. Then, either
    \begin{enumerate}[\normalfont(i)]
        \item $\ker \pi|_H\ne 1$; or
        \item $\ker \pi|_H= 1$.
    \end{enumerate}
    
    If (i) holds, then there exists a non-trivial element $1\ne b\in L^{h_{k+1}}$ such that
    $$(1,\dotsc,1,b)\in H\le L\times L^{h_1}\times\dotsb \times L^{h_{k+1}},$$
    and thus also a non-trivial element $1\ne a\in L$ such that
    $$h_a:=(1,\dotsc,1,b)^{h_{k+1}^{-1}}=(a,1,\dotsc, 1)\in H\le L\times L^{h_1}\times\dotsb \times L^{h_{k+1}}.$$
    In this case, the statement holds with $i=k+1$ and $N:=\langle a\rangle^L\ne 1$. Indeed, as $H$ is a subdirect product, we have
    $$N\times 1\times\dotsb \times 1\le H.$$
    
    If (ii) holds instead, then the rightmost component is completely determined by the first $k+1$ components. Note that in this case $\pi(H)\ne 1$, as $H\ne 1$ and $\ker \pi|_H=1$. Then, we may apply the inductive hypothesis to $\pi(H)$ to obtain a maximal $i\in \{0,1,\dotsc,k\}$ and a non-trivial normal subgroup $1\ne N\le L$, such that for each $a\in N$, there exist unique $b_1,\dotsc,b_{k-i}$ satisfying that
    $$h_a:=(a,1,\overset{i}{\ldots},1,b_1,\dotsc,b_{k-i})\in \pi(H)\le L\times L^{h_1}\times\dotsb\times  L^{h_{k}}.$$
    Since $\ker\pi|_H=1$, there is a unique lift of $h_a$ to $H$, i.e. for each $a\in N$, there exists a unique $b_{k+1-i}\in L^{h_{k+1}}$ such that
    $$\pi^{-1}(h_a)=(a,1,\overset{i}{\ldots},1,b_1,\dotsc,b_{k-i},b_{k+1-i})\in H\le L\times L^{h_1}\times\dotsb\times  L^{h_{k}}\times L^{h_{k+1}}.$$
    Therefore, in this case, the statement also holds, with the same $i$ and $1\ne N\trianglelefteq L$ from the inductive hypothesis.
\end{proof}

\subsection{Positive-dimensional groups}

We shall need the following standard lemma; see for instance \cite[Lemma 4]{pro-c}. We provide a short proof for the convenience of the reader:

\begin{lemma}
\label{lemma: normal subgroups and rists}
    Let $G\le \mathrm{Aut}~T$ and let us consider a non-trivial normal subgroup $1\ne N\trianglelefteq G$. If $v\in T$ is moved by $G$, then $N\ge \mathrm{rist}_G(v)'$.
\end{lemma}
\begin{proof}
    Consider $a,b\in \mathrm{rist}_G(v)$ for some vertex $v\in T$ moved by some element $g\in G$. Then
    $$[a,b]=[[g,a],b]\in [[N,G],G]\le N,$$
    since $a^g$ commutes with both $a$ and $b$, as $g$ moves $v$ so $\mathrm{rist}_G(v)\cap \mathrm{rist}_G(v^g)=1$.
\end{proof}

Now, we have all the tools that we need to prove \cref{Theorem: linear groups are zero-dimensional}. We will apply the lifting lemma to positive-dimensional level-transitive subgroups of $\mathrm{Aut}~T$ to lift a non-commuting representation of $V_n$ from a projection on the tree to the original group and prove \cref{Theorem: linear groups are zero-dimensional}:

\begin{proof}[Proof of \cref{Theorem: linear groups are zero-dimensional}]
Let $G\le_c \mathrm{Aut}~T$ be a level-transitive closed subgroup of $\mathrm{Aut}~T$ with positive upper-box dimension in $\mathrm{Aut}~T$. If $G\le \mathrm{GL}_k(R)$ for some $k\ge 1$ then~$G$ does not admit non-commuting representations of $V_n$ for $n>k^2$ by \cref{lemma: minimal number of representations}. Thus, if we show that $G$ admits a non-commuting representation of $V_n$ for every $n\ge 1$, the group $G$ cannot be linear over $R$ and \cref{Theorem: linear groups are zero-dimensional} is proved. 

If $G$ is weakly branch, then just consider a level of the tree $T$ with at least $n$ vertices and take two non-commuting elements from $n$ distinct rigid vertex stabilizers. This can be done because the rigid vertex stabilizers of a weakly branch group are not abelian by \cite[Lemma 2.17]{MaximalDominik} (they do not even satisfy any group law by \cite[Theorem 1]{Abert}) and two distinct rigid vertex stabilizers of vertices at the same level commute.

Let us assume $G$ is not weakly branch. Still, by \cref{proposition: rists of positive dimensional subgroups}, almost all the projections of $G$ on $T$ are weakly branch. Thus, it will be enough to show that one can lift these projections in such a way that the lifts commutes if and only if the projections commute.

For any $n\ge 1$, let us write $w_n:=1\overset{n}{\ldots} 1$, i.e. $w_n$ is the leftmost vertex at the $n$th level of $T$. Let $N\ge 1$ be the natural number given by \cref{proposition: rists of positive dimensional subgroups}, i.e. the natural number $N$ such that all the projections of $G$ at level~$N$ and below are weakly branch. We set 
$$H:=\psi_N(\mathrm{St}_G(N))$$
and define the subgroup $L\le \psi_N(\mathrm{St}(N))$ as
$$L:=H\cap \psi_N(\mathrm{rist}(w_N))=\psi_N(\mathrm{St}_G(N))\cap \psi_N(\mathrm{rist}(w_N)).$$
Now, by level-transitivity of $G$, we may choose $h_1,\dotsc,h_k\in G$, with $k:=N_N-1$, such that
$$w_N^{h_i}=v_i$$
for each $1\le i\le k$, where $v_1,\dotsc,v_k$ are the distinct vertices in $\mathcal{L}_N\setminus\{w_N\}$. In order to simplify notation, we shall abuse of notation and write $h_i$ for both $h_i$ and its image via~$\psi_N$. Then
$$L^{h_i}=H^{h_i}\cap \psi_N(\mathrm{rist}(w_N^{h_i}))=\psi_N(\mathrm{St}_G(N))\cap \psi_N(\mathrm{rist}(v_i))$$
for each $1\le i\le k$, and we get that
$$H\le L\times L^{h_1}\times\dotsb L^{h_k}$$
is a subdirect product. Therefore, we may apply the lifting lemma (\cref{lemma: lifting lemma}) to obtain a maximal $i\in \{0,1,\dotsc,k\}$ and a non-trivial normal subgroup $1\ne N\trianglelefteq L$ such that for each $a\in N$, there exist unique $b_1,\dotsc,b_{k-i}$ satisfying that
\begin{align}
\label{align: lift}
\psi_N(g_a)&:=(a,1,\overset{i}{\ldots},1,b_1,\dotsc,b_{k-i})\in H.
\end{align}
Since the $b_1,\dotsc,b_{k-i}$ in \cref{align: lift} are uniquely determined by $i$ and $a$, it is enough to obtain a non-commuting representation of $V_n$ in $N$. Indeed, for any pair $a,b\in N$, the elements $g_a$ and $g_b$ commute if and only if the corresponding elements $a$ and $b$ commute.

Now, note that $L$ is a finite index subgroup of $G_{w_N}$. By the choice of $w_N$, the group $G_{w_N}$ is weakly branch, so for every $v\in T_{w_N}$ we have that the rigid vertex stabilizer $\mathrm{rist}_{G_{w_N}}(v)$ is infinite. Hence  
$$\mathrm{rist}_{L}(v)=\mathrm{rist}_{G_{w_N}}(v)\cap L\ne 1.$$
Thus, for each $v\in T_{w_N}$, the subgroup $L$ moves at least one vertex in $T_v$. By \cref{lemma: normal subgroups and rists}, for each vertex $u\in T_{w_N}$ moved by $L$ we have
$$N\ge \mathrm{rist}_L(u)'.$$
Then, for any $n\ge 1$, we may find $n$ distinct vertices $v_1,\dotsc,v_n\in T_{w_N}$ such that
$$N\ge \prod_{i=1}^n\mathrm{rist}_L(v_i)',$$
(for example by choosing $n$ vertices at the same level of $T_{w_N}$ and considering one moved by $L$ in each of the corresponding subtrees). By \cite[Lemma 2.17]{MaximalDominik} (note that Francoeur's proof does not make use of level-transitivity, so his proof also holds for our subgroup $L$) each non-trivial rigid stabilizer $\mathrm{rist}_L(v_i)$ is not solvable. In particular, for each $1\le i\le n$, we may find non-commuting elements
$$a_i,b_i\in \mathrm{rist}_L(v_i)'\le N.$$
Therefore, for every $1\le i,j\le n$, we have
$$a_ia_j=a_ja_i\quad \text{and}\quad b_ib_j=b_jb_i,$$
and 
$$a_ib_j=b_ja_i\quad\text{if and only if}\quad i\ne j;$$
i.e. we obtain a non-commuting representation of $V_n$ in $N$.
\end{proof}



\bibliographystyle{unsrt}

\begin{thebibliography}{1}

\bibitem{Abert}
M. Abért, Group laws and free subgroups in topological groups, \textit{Bull. Lond. Math. Soc.} \textbf{37} (2005), 525--534.

\bibitem{AbertGraph}
M. Abért, Representing graphs by the non-commuting relation, \textit{Publ. Math. Debrecen} \textbf{69 (3)} (2006), 261--269.

\bibitem{AbertVirag}
M. Abért and B. Virág, Dimension and randomness in groups acting on rooted trees, \textit{J. Amer. Math. Soc.} \textbf{18} (2005), 157--192.

\bibitem{BarneaShalev}
Y. Barnea and A. Shalev, Hausdorff dimension, pro-$p$ groups, and Kac-Moody algebras, \textit{Trans. Amer. Math. Soc.} \textbf{349} (1997), 5073--5091.

\bibitem{BartholdiHausdorff}
L. Bartholdi, Branch rings, thinned rings, tree enveloping rings, \textit{Israel J. Math.} \textbf{158} (2006), 93--139.

\bibitem{TitsAlternative}
 E. Breuillard and T. Gelander, A topological Tits alternative, \textit{Ann. Math.} \textbf{166} (2007), 427--474.

 \bibitem{Bridy}
A. Bridy, R. Jones, G. Kelsey, and R. Lodge, Iterated monodromy groups of rational functions and periodic points over finite fields, \textit{Math. Ann.} \textbf{390(1)} (2024), 439--475.

\bibitem{BostonNH}
N. Boston, $p$-adic Galois representations and pro-$p$ Galois groups, in:
\textit{New Horizons in Pro-p Groups}, Birkhäuser Boston, MA \textbf{1}, 2000.

\bibitem{Boston}
N. Boston, Some cases of the Fontaine-Mazur conjecture, II, \textit{J. Number Theory} \textbf{75 (2)} (1999), 161--169.


\bibitem{Calegari}
F. Calegari, Even Galois representations and the Fontaine-Mazur conjecture, \textit{Invent. Math.} \textbf{185} (2011), 1--16.

 \bibitem{MicroSupported1}
P.-E. Caprace and A. Le Boudec (with an appendix by D. Francoeur), Commensurated subgroups and micro-supported actions, \textit{J. Eur. Math. Soc.} \textbf{25} (2023), 2251--2294.

 \bibitem{MicroSupported2}
P.-E. Caprace, C.\,D. Reid and G.\,A. Willis, Locally normal subgroups of totally disconnected groups. Part II: Compactly generated simple groups, \textit{Forum Math. Sigma} \textbf{5} (2017), art. e12, 89 pp.

\bibitem{QuestionAbertVirag}
J. Fariña-Asategui, On a question of Abért and Virág, arXiv preprint: 2505.23142.

\bibitem{RestrictedSpectra}
J. Fariña-Asategui, Restricted Hausdorff spectra of $q$-adic automorphisms, \textit{Adv. Math.} \textbf{472} (2025), 110294.

\bibitem{JorgeSantiFPP}
J. Fariña-Asategui and S. Radi, On the fixed-point proportion of self-similar groups, arXiv preprint: 2503.00185.

\bibitem{Ferraguti}
A. Ferraguti, A. Ostafe and U. Zannier, Cyclotomic and abelian points in backward orbits of rational functions, \textit{Adv. Math.} \textbf{438} (2024), 109463.

\bibitem{FontaineMazur}
J.-M. Fontaine and B. Mazur, Geometric Galois representations, \textit{Elliptic Curves and Modular Forms}, Proceedings of a conference held in Hong Kong, December 18--21 (1993), International Press, 1997.

\bibitem{MaximalDominik}
D. Francoeur, On maximal subgroups of infinite index in branch and weakly branch groups, \textit{J. Algebra}, \textbf{560}
(2020), 818--851. 

\bibitem{pro-c} 
A. Garrido and J. Uria-Albizuri, Pro-$\mathcal{C}$ congruence properties for groups of rooted tree automorphisms, \textit{Arch. Math.} \textbf{112} (2019), 123--137.


\bibitem{JonesAMS}
R. Jones, Fixed-point-free elements of iterated monodromy groups, \textit{Trans. Amer. Math. Soc.} \textbf{367 (3)} (2015), 2023--2049.

\bibitem{JonesArboreal}
R. Jones, Galois representations from pre-image trees: an arboreal survey, \textit{Publications mathématiques de Besançon. Algèbre et théorie des nombres} (2013), 107--136.

\bibitem{JonesComp}
R. Jones, Iterated Galois towers, their associated martingales, and the $p$-adic Mandelbrot set, \textit{Compos. Math.} \textbf{143 (5)} (2007), 1108--1126. 

\bibitem{JonesLMS}
R. Jones, The density of prime divisors in the arithmetic dynamics of quadratic polynomials, \textit{J. Lond. Math. Soc.} \textbf{78 (2)} (2008), 523--544.


\bibitem{Juul}
J. Juul, P. Kurlberg, K. Madhu, and T.\,J. Tucker, Wreath products and proportions of periodic points, \textit{Int. Math. Res. Not. IMRN} \textbf{13} (2016), 3944--3969.

\bibitem{KionkeSchesler}
S. Kionke and E. Schesler, On representations of direct products and the bounded generation property of branch groups, \textit{Arch. Math.} \textbf{120} (2023), 449--455.

\bibitem{Kisin}
M. Kisin, The Fontaine-Mazur conjecture for GL2, \textit{J. Amer. Math. Soc.} \textbf{22 (3)} (2009), 641--690.

\bibitem{Odoni1}
R.\,W.\,K. Odoni, On the prime divisors of the sequence $w_{n+1} = 1 + w_1 \dots w_n$, \textit{J. Lond. Math. Soc.} \textbf{32(2)} (1980), 1--11.

\bibitem{Odoni2}
R.\,W.\,K. Odoni, The Galois theory of iterates and composites of polynomials, \textit{Proc. Lond. Math. Soc.}  \textbf{51 (3)} (1985), 385--414.

\bibitem{Pan}
L. Pan, The Fontaine-Mazur conjecture in the residually reducible case, \textit{J. Amer. Math. Soc.} \textbf{35 (4)} (2022), 1031--1169.

\bibitem{SkinnerWiles}
C.\,M. Skinner and A.\,J. Wiles, Residually reducible representations and modular forms, \textit{Publ. Math. Inst. Hautes Études Sci.} \textbf{89} (1999), 5--126.






\end{thebibliography}

\end{document}